\documentclass[12pt]{amsart}
\usepackage{txfonts}
\usepackage{subfig}
\usepackage{amscd,amssymb,mathabx}
\usepackage[arrow,matrix,graph,frame,poly,arc,tips]{xy}
\usepackage{graphicx,calrsfs}
\usepackage{mathtools}
\usepackage{tikz}
\usetikzlibrary{decorations.markings}
\usetikzlibrary{shapes}
\usetikzlibrary{arrows}
\usetikzlibrary{backgrounds}
\usetikzlibrary{calc}
\usepackage{mdwlist}
\usepackage{multirow}
\usepackage{enumerate}
\usepackage{verbatim}
\usepackage{etoolbox}
\AtEndEnvironment{proof}{\setcounter{claim}{0}}

\DeclarePairedDelimiter{\form}{\langle}{\rangle}

\newcommand\ba{\begin{align*}}
\newcommand\ea{\end{align*}}
\newcommand\be{\begin{enumerate}}
\newcommand\ee{\end{enumerate}}
\newcommand\bp{\begin{proof}}
\newcommand\ep{\end{proof}}
\newcommand\bpp{\begin{prop}}
\newcommand\epp{\end{prop}}
\newcommand\bpb{\begin{prob}}
\newcommand\epb{\end{prob}}
\newcommand\bd{\begin{defn}}
\newcommand\ed{\end{defn}}
\newcommand\bh{\begin{hint}}
\newcommand\eh{\end{hint}}

\newcommand\stab{\mathrm{Stab}}

\newcommand\bN{\mathbb{N}}

\newcommand\bR{\mathbb{R}}
\newcommand\R{\mathbb{R}}
\newcommand\bQ{\mathbb{Q}}

\newcommand\bZ{\mathbb{Z}}
\newcommand\Z{\mathbb{Z}}

\renewcommand\AA{\mathcal{A}}
\newcommand\BB{\mathcal{B}}
\newcommand\CC{\mathcal{C}}
\newcommand\FF{\mathcal{F}}

\newcommand\DD{\mathcal{D}}

\newcommand\KK{\mathcal{K}}

\newcommand\supp{\operatorname{supp}}

\newcommand\SL{\operatorname{SL}}
\newcommand\Diffb{\operatorname{Diff}_+^{1+\mathrm{bv}}}
\newcommand\Cb{C^{1+\mathrm{bv}}}

\newcommand\gam{\Gamma}
\newcommand\Mod{\operatorname{Mod}}
\newcommand\PSL{\operatorname{PSL}}

\DeclareMathOperator\Homeo{Homeo}

\newcommand\sse{\subseteq}
\newcommand\co{\colon}

\DeclareMathOperator\Fix{Fix}

\DeclareMathOperator\Diff{Diff}

\newcommand\var{\operatorname{var}}
\newcommand\rot{\operatorname{rot}}
\newcommand\Cbv{C^{1+\mathrm{bv}}}
\DeclareMathOperator\Per{Per}

\renewcommand{\MR}[1]
{\href{http://www.ams.org/mathscinet-getitem?mr=#1}{MR#1}}

\def\thetitle{{Free products and the algebraic structure of diffeomorphism groups}}
\def\theauthors{{Sang-hyun Kim and Thomas Koberda}}
\usepackage{hyperref}
\hypersetup{
   colorlinks=false,
   plainpages,
   urlcolor=black,
   linkcolor=black
   pdftitle=   \thetitle,
   pdfauthor=  {\theauthors}
}

\theoremstyle{theorem}
\newtheorem{thm}{Theorem}[section]
\newtheorem{lem}[thm]{Lemma}
\newtheorem{cor}[thm]{Corollary}
\newtheorem{prop}[thm]{Proposition}

\newtheorem{que}[thm]{Question}
\newtheorem*{claim*}{Claim}

\newtheorem{claim}{Claim}

\theoremstyle{remark}

\newtheorem{rem}[thm]{Remark}

\theoremstyle{definition}
\newtheorem{defn}[thm]{Definition}
\newtheorem{prob}{Problem}[section]

\begin{document}
\title\thetitle
\date{\today}

\author[S. Kim]{Sang-hyun Kim}
\address{Department of Mathematical Sciences, Seoul National University, Seoul, Korea}
\email{s.kim@snu.ac.kr}
\urladdr{http://cayley.kr}

\author[T. Koberda]{Thomas Koberda}
\address{Department of Mathematics, University of Virginia, Charlottesville, VA 22904-4137, USA}
\email{thomas.koberda@gmail.com}
\urladdr{http://faculty.virginia.edu/Koberda}

\begin{abstract}
Let $M$ be a compact one--manifold,
and let $\Diffb(M)$ denote the group of $C^1$ orientation preserving diffeomorphisms of $M$ whose first derivatives have bounded variation.
We prove that if $G$ is a group which is not virtually metabelian, then $(G\times\Z)*\Z$ is not realized as a subgroup of $\Diffb(M)$.
This  gives the first examples of finitely generated groups $G,H\le \Diff_+^\infty(M)$ such that $G\ast H$ does not embed into $\Diffb(M)$.
By contrast, 
for all countable groups $G,H\le\Homeo^+(M)$ there exists an embedding $G\ast H\to \Homeo^+(M)$.
We deduce that many common groups of homeomorphisms do not embed into $\Diffb(M)$, for example the free product of $\bZ$ with Thompson's group $F$.
We also complete the classification of right-angled Artin groups which can act smoothly on $M$
and in particular, recover the main result of a joint work of the authors with Baik~\cite{BKK2014}. 
Namely, a right-angled Artin group $A(\gam)$ either admits a faithful $C^{\infty}$ action on $M$, or $A(\gam)$ admits no faithful $\Cb$ action on $M$.  In the former case, $A(\gam)\cong\prod_i G_i$
where $G_i$ is a free product of free abelian groups.
Finally, we develop a hierarchy of right-angled Artin groups, with the levels of the hierarchy corresponding to the number of semi-conjugacy classes of possible actions of these groups on $S^1$.
\end{abstract}

\maketitle

\section{Introduction}

Let $M$ be a compact one--manifold. In this article, we study the algebraic structure of the group $\Diffb(M)$, where here $\Diffb(M)$ denotes the group of $C^1$ diffeomorphisms of $M$ whose derivatives have bounded variation.
Specifically, we consider the restrictions placed on subgroups of $\Diffb(M)$ by the $\Cb$ regularity assumption. Our main result implies that there is a large class $\AA_0 $ of finitely generated subgroups of $\Diff_+^\infty(M)$ such that for all $G,H\in\AA_0$, the free product $G\ast H$ can never be realized as a subgroup of $\Diffb(M)$; see Corollary~\ref{cor:classification}.

As a corollary, we complete a program initiated by Baik and the authors in~\cite{BKK2014,BKK2016} to decide which right-angled Artin groups admit faithful $C^{\infty}$ actions on a compact one--manifold, and exhibit many classes of finitely generated subgroups of $\Homeo^+(M)$ which cannot be realized as subgroups of $\Diffb(M)$.

\subsection{Statement of results}

Unless otherwise noted, $M$ will denote a compact one--manifold.
 That is to say, $M$ is a finite union of disjoint closed intervals $I=[0,1]$ and circles $S^1=\bR/\bZ$.
An \emph{action} on a one--manifold is always assumed to mean 
an orientation preserving action. 
For each $0\le r\le\infty$ or $r=\omega$, we let $\Diff^r_+(M)$ denote the group of orientation preserving $C^r$ (analytic if $r=\omega$) diffeomorphisms.
We write $\Homeo^+(M)=\Diff^0_+(M)$.
Our main result is the following:

\begin{thm}\label{thm:main}
If $G$ is a group which is not virtually metabelian, then the group $(G\times\Z)*\Z$ admits no faithful $C^{1+\mathrm{bv}}$ action on $M$.
\end{thm}

Theorem~\ref{thm:main} clearly implies that $(G\times\Z)*\Z$ does not embed into $\Diff^r_+(M)$ for every level of regularity $r\geq 2$.
 A key step in our proof of Theorem~\ref{thm:main} is the following result on $C^1$--smoothability:
\begin{thm}[Theorem~\ref{t:tech-main}]\label{thm:tech-main-intro}
Let $X\in\{I,S^1\}$,
and let
 $a,b,t\in\Diff^1_+(X)$.
 If
\[\supp a\cap\supp b=\varnothing,\] 
then the group $\form{a,b,t}$  is not isomorphic to $\bZ^2\ast\bZ$.
\end{thm}

We have stated Theorem~\ref{thm:main} as above for clarity and concision, though several stronger statements can be deduced from the proof we give. In the particular case of finitely generated groups, we note the following, which should be compared to Corollary~\ref{cor:cnt}:

\begin{cor}\label{cor:fg1}
If $G$ is a finitely generated group which is not virtually abelian, then $(G\times\Z)*\Z$ admits no faithful $\Cb$ action on $M$. 
\end{cor}
The motivation for proving Theorem~\ref{thm:main} came from investigating right-angled Artin subgroups of $\Diffb(M)$ (cf. Corollary~\ref{cor:classification} below). The simplest right-angled Artin group to which Theorem~\ref{thm:main} applies is $(F_2\times\Z)*\Z$:

\begin{cor}\label{cor:FtimesZ}
The group $(F_2\times\Z)*\Z$ is not a subgroup of $\Diffb(M)$.
\end{cor}

Since the hypotheses on Theorem~\ref{thm:main} are relatively weak, there are many finitely generated subgroups of $\Homeo^+(M)$ which can be thus shown to admit no faithful $\Cb$ actions on a compact manifold. Recall that \emph{Thompson's group $F$} is a group of piecewise linear homeomorphisms of the interval, with dyadic break points and with all slopes being powers of two. \emph{Thompson's group $T$} is the analogous group defined for the circle. A famous result of Ghys--Sergiescu~\cite{GS1987} says that the standard actions of $T$ and $F$ are topologically conjugate into a group of $C^{\infty}$ diffeomorphisms of the circle and of the interval, respectively. However, we have the following:

\begin{cor}\label{cor:Thompson}
The groups $F*\Z$ and $T*\Z$ are not subgroups of $\Diffb(M)$.
\end{cor}

In the vein of Corollary~\ref{cor:Thompson}, we do not know the answer to the following:

\begin{que}\label{que:thompson-critical}
Are the groups $F*\Z$ and $T*\Z$ subgroups of $\Diff_+^1(I)$ and $\Diff_+^1(S^1)$ respectively? If so, what is their optimal regularity, i.e. the supremum of the H\"older continuity exponent $\tau\in [0,1)$ such that these groups embed in $\Diff_+^{1+\tau}$ for the relevant manifolds (cf.~\cite{JNR2018,KK2017crit})?
\end{que}

Returning to the original motivation, Theorem \ref{thm:main} combined with results of Farb--Franks and Jorquera completes the classification of right-angled Artin groups admitting actions of various regularities on compact one--manifolds. Before stating the result, we define some terminology. We write $\gam$ for a finite simplicial graph with vertex set $V(\gam)$ and edge set $E(\gam)$. The \emph{right-angled Artin group} (or \emph{RAAG}, for short) on $\gam$ is defined as \[A(\gam)=\langle V(\gam)\mid [v,w]=1 \textrm{ if and only if }\{v,w\}\in E(\gam)\rangle.\]

A subgraph $\Lambda$ of a graph $\gam$ is called a \emph{full} subgraph if $\Lambda$ is spanned by the vertices of $\Lambda$. That is, two vertices in $\Lambda$ are adjacent if and only if they are adjacent in $\gam$.
A simplicial graph $\gam$ is called \emph{$P_4$--free} if no full subgraph of $\gam$ is isomorphic to a path $P_4$ on four vertices. Such graphs are often called \emph{cographs}~\cite{CLB1981,KK2013}. We write $\mathcal{K}$ for the class of cographs. It is well--known that cographs can be fit into a hierarchy which is defined as follows:
\begin{enumerate}
\item
The class $\mathcal{K}_0$ consists of a single vertex;
\item
If $n\geq 1$ is odd then $\mathcal{K}_n$ is obtained by taking $\mathcal{K}_{n-1}$ together with finite joins of elements in $\mathcal{K}_{n-1}$;
\item
If $n\geq 2$ is even then $\mathcal{K}_n$ is obtained by taking $\mathcal{K}_{n-1}$ together with finite disjoint unions of elements in $\mathcal{K}_{n-1}$.
\end{enumerate}
Here, a \emph{join} of two simplicial graphs $X$ and $Y$ is a simplicial graph consisting of the disjoint union of $X$ and $Y$, together with an edge of the form $\{x,y\}$ for every vertex $x$ of $X$ and every vertex $y$ of $Y$.

We have that $\mathcal{K}_0\subset\mathcal{K}_1\subset\KK_2\subset\cdots$ and \[\mathcal{K}=\bigcup_i\mathcal{K}_i.\] 
Note that join and disjoint union correspond to direct product and free product respectively, so that right-angled Artin groups on cographs are exactly the smallest class of groups containing $\Z$, which is closed under finite direct products, and which is closed under finite free products. The reader will observe that if $\gam\in\mathcal{K}_0$ then $A(\gam)\cong\Z$. Similarly, if $\gam\in\mathcal{K}_1$ then $A(\gam)$ is free abelian, and if $\gam\in\mathcal{K}_2$ then $A(\gam)$ is a free product of free abelian groups.
If $\gam\in\mathcal{K}_3$ then $A(\Gamma)$ can be written as \[A(\Gamma)=\prod_{i=1}^m G_i,\]
  where each $G_i$ is a free product of free abelian groups.
In~\cite{BKK2016}, Baik and the authors proved that if $A(\gam)$ admits an injective homomorphism into $\Diffb(M)$, then $\gam\in\KK$. We will deduce a strengthening of this result, using Theorem~\ref{thm:main}.


\begin{cor}[cf. \cite{BKK2016}]\label{cor:classification}
Let $A(\gam)$ be a right-angled Artin group. 
\begin{enumerate}
\item (see \cite{FF2003,Jorquera})
There exists an injective homomorphism $A(\gam)\to\Diff_+^1(M)$; thus, $A(\gam)$ admits a faithful $C^1$ action of $M$.
\item
If there exists an injective homomorphism $A(\gam)\to\Diffb(M)$ then $\gam\in\mathcal{K}_3$; conversely, if $\gam\in\mathcal{K}_3$ then $A(\gam)\le\Diff_+^{\infty}(M)$.
\end{enumerate}
\end{cor}

In particular, if we define
\[
\AA_0=\{A(\Gamma)\mid \Gamma\in\KK_3\setminus\KK_2\}\]
then each $G\in\AA_0$ embeds into $\Diff_+^\infty(M)$,
but for all $G,H\in\AA_0$ the group $G\ast H$ never embeds into $\Diff_+^{1+\mathrm{bv}}(M)$.
By contrast, we have the following, which is well--known from several contexts.

\begin{prop}[cf. \cite{KM1996,Rivas2012JAlg,BS2015}]\label{prop:free prod}
The class of countable subgroups of $\Homeo^+(M)$ is closed under finite free products.
\end{prop}

In the same spirit of Question~\ref{que:thompson-critical} and the authors' paper~\cite{KK2017crit}, we have the following:

\begin{que}
Let $\gam\notin\mathcal{K}_3$. What is the supremum of $\tau\in [0,1)$ for which $A(\gam)$ embeds in $\Diff_+^{1+\tau}(M)$? Does $\tau$ depend on $\gam$?
\end{que}

Let $S$ be an orientable surface of genus $g$ and with $n$ punctures or boundary components. We say that $S$ is \emph{sporadic} if \[3g-3+n\leq 1.\] We write $\Mod(S)$ for the mapping class group of $S$, i.e. $\Mod(S)=\pi_0(\Homeo^+(S))$. Using the main result of~\cite{Koberda2012}, we immediately recover the following result as a corollary of Theorem~\ref{thm:main}:

\begin{cor}[cf.~\cite{BKK2016}]
Let $M$ be a compact one--manifold, and let $S$ be an orientable finite-type surface. Then there exists a finite index subgroup $G\le\Mod(S)$ such that $G\le\Diffb(M)$ if and only if $S$ is sporadic.
\end{cor}

Theorem \ref{thm:main} allows us to build a hierarchy on right-angled Artin groups, whose levels correspond to right-angled Artin groups with more or fewer ``dynamically different" actions on the circle. Roughly speaking, two group actions \[\rho_1,\rho_2\colon G\to\Homeo^+(S^1)\] are \emph{semi--conjugate} (or, \emph{monotone-equivalent}) if there exists 
another action \[\rho\co G\to\Homeo^+(S^1)\]
and monotone degree one maps $h_i\colon S^1\to S^1$ such that
\[h_i\circ \rho=\rho_i\circ h_i\] for each $i=1,2$. See~\cite{Ghys1987,CD2003IM,Ghys2001,Mann-hb,BFH2014,KKMj2016} for instance, and the many references therein.
A \emph{projective action} of a group $G$ is a representation
\[
\rho\co G\to \PSL_2(\bR),\]
where $\PSL_2(\R)$ sits inside of $\Homeo^+(S^1)$ as the group of projective analytic diffeomorphisms of $S^1$. 


\begin{cor}\label{cor:conj}
Let $A(\gam)$ be a right-angled Artin group.
\begin{enumerate}
\item
If $\gam\in\mathcal{K}_2$, then $A(\gam)$ admits uncountably many distinct semi--conjugacy classes of faithful orientation preserving projective actions on $S^1$;
\item
If $\gam\in\mathcal{K}_3\setminus\mathcal{K}_2$ then any faithful orientation preserving $\Cb$ action of $A(\gam)$ on $S^1$ has a periodic point and no dense orbits and hence admits at most countably many distinct semi--conjugacy classes of $\Cb$ actions on $S^1$;
\item
If $\gam\notin\mathcal{K}_3$ then $A(\gam)$ admits no faithful $\Cb$ action on $S^1$.
\end{enumerate}
\end{cor}

In the case of analytic actions on a compact connected one--manifold $M$, one has the following result of Akhmedov and Cohen:

\begin{thm}[See~\cite{AC2015TA}]
The right-angled Artin group $A(\gam)$ embeds into $\Diff^\omega(M)$ if and only if $\Gamma\in\KK_2$, i.e. $A(\gam)$ decomposes as a free product of free abelian groups.
\end{thm}

\subsection{Notes and references}

This paper
reveals some of the subtlety of the interplay between algebra and regularity in diffeomorphism groups of one--manifolds. 
Our paper arose during the effort to complete the classification of right-angled Artin subgroups of $\Diff_+^{\infty}(S^1)$
in the spirit of~\cite{BKK2016}.
The essential content of this paper is Lemma~\ref{l:main}, which exhibits an explicit element of the kernel of any given $\Cb$ action of $(G\times\Z)*\Z$ action on a compact one--manifold.
Since a group of the form $(G\times\Z)*\Z$ is generally simpler than the right-angled Artin groups considered in~\cite{BKK2016}, it is more difficult to find elements in the kernel of a given action, and therefore we develop more sophisticated tools here. We note that our main result does subsume the main result of~\cite{BKK2016}. Indeed, the main result of~\cite{BKK2016} is that there is no injective homomorphism $A(P_4)\to\Diffb(M)$, where here \[A(P_4)=\langle a,b,c,d\mid [a,b]=[b,c]=[c,d]=1\rangle.\] The group $A(P_4)$ contains a copy of $(F_2\times\Z)*\Z$, which cannot embed in $\Diffb(M)$ by Corollary~\ref{cor:FtimesZ}. An explicit embedding of $(F_2\times\Z)*\Z$ into $A(P_4)$ is given by $\langle a,b,c,dad^{-1}\rangle\le A(P_4)$ (see~\cite{KK2013} for a discussion of this fact).

The program completed by Corollary~\ref{cor:classification} fully answers a question raised in a paper of M. Kapovich (attributed to Kharlamov) as to which right-angled Artin groups admit faithful $C^{\infty}$ actions on the circle~\cite{Kapovich2012}.

Right-angled Artin subgroups of diffeomorphism groups of one--manifolds find an analogue in right-angled Artin subgroups of linear groups. It is well--known that right-angled Artin groups are always linear over $\Z$ and hence admit injective homomorphisms into $\SL_n(\Z)$~\cite{Humphries1994,HW1999,DJ2000}. For $\SL_3(\Z)$, it is still unclear which right-angled Artin groups appear as subgroups. Long--Reid~\cite{LR2011} showed that $F_2\times\Z$ is not a subgroup of $\SL_3(\Z)$, which implies that any right-angled Artin subgroup of $\SL_3(\Z)$ is a free product of free abelian groups of rank at most two. It is currently unknown whether or not $\Z^2*\Z$ is a subgroup of $\SL_3(\Z)$.

A consequence of the technical work behind Theorem~\ref{thm:main} is a certain criterion to prove that a group contains a lamplighter subgroup. See Lemma~\ref{l:recursive} and Proposition~\ref{t:tech-main2} for precise statements.

A crucial step in our proof of the main theorem is a $C^1$--rigidity result, which is Theorem~\ref{t:tech-main}. 
We note that Thurston~\cite{Thurston1974Top}, Calegari~\cite{Calegari2008AGT}, Navas~\cite{Navas2008GAFA} and Bonatti, Monteverde, Navas and Rivas~\cite{BMNR2017MZ} explored various remarkable $C^1$--rigidity results. 
In particular, a $C^1$--rigidity result on Baumslag--Solitar groups in~\cite{BMNR2017MZ} was employed in a very recent paper by Bonatti, Lodha and Triestino~\cite{BLT2017},
to produce certain piecewise affine homeomorphism groups of $\bR$ which do not embed into $\Diff^1_+(I)$.

As for other classes of groups of homeomorphisms which cannot be realized as groups of $\Cb$ diffeomorphisms, Corollary~\ref{cor:Thompson} is a complement to a result of the second author with Lodha~\cite{KL2017}, in which they show that certain ``square roots" of Thompson's group $F$ may fail to act faithfully by $\Cb$ diffeomorphisms on a compact one--manifold, even though they are manifestly groups of homeomorphisms of these manifolds. Thus, the Ghys--Sergiescu Theorem appears to place $F$ and $T$ at the cusp of smoothability in the sense that even relatively minor algebraic variations on $F$ and on $T$ fail to be smoothable.

Finally, we remark on the optimality of the differentiability hypothesis. 
Corollary~\ref{cor:classification} shows that every right-angled Artin group can act faithfully by $C^1$ diffeomorphisms, but nearly none of them can act by $\Cb$ diffeomorphisms. 
As for Corollary~\ref{cor:not closed} and Proposition~\ref{prop:free prod}, we have the following result (based on~\cite{BMNR2017MZ} and on a suggestion by Navas), which is proved in the authors' recent manuscript~\cite{KK2017crit}:

\begin{prop}[\cite{KK2017crit}]\label{prop:c1freeproduct}
Let $M\in\{I,S^1\}$, and let $BS(1,2)$ be the Baumslag--Solitar group with the presentation $\langle s,t\mid sts^{-1}=t^2\rangle$. Then the group $(BS(1,2)\times\Z)*\Z$ is not a subgroup of $\Diff_+^1(M)$. In particular, the class of finitely generated subgroups of $\Diff_+^1(M)$ is not closed under taking finite free products.
\end{prop}

\section{Background on one--dimensional smooth dynamics}
We very briefly summarize the necessary background from one--dimensional dynamics. The reader may also consult~\cite{BKK2016}, parts of which we repeat here, and where we direct the reader for proofs.

\subsection{Poincar\'e's theory of rotation numbers}

Let $f\in\Homeo^+(S^1)$, and let $\tilde f\co \bR\to\bR$ be  an arbitrary lift of $f$.
Then the \emph{rotation number} of $f$ is defined as
\[
\rot f = \lim_{n\to\infty}\frac{\tilde f^n(x)}{n}\in \bR/\bZ=S^1\]
where $x\in \bR$. Then $\rot f$ is well-defined, and independent of the choice of a lift $\tilde f$ and a base point $x\in \bR$; see~\cite{Navas2011} for instance. 
The set of periodic points of $f$ is denoted as $\Per f$.
Let us record some elementary facts.
\begin{lem}\label{l:inv}
For $f\in\Homeo^+(S^1)$, the following hold.
\be
\item
$\rot(f) =0$ if and only if $\Fix f\ne\varnothing$.
\item
$\rot(f)\in\bQ$ if and only if $\Per f\ne\varnothing$.
\item\label{p:inv}
If $x\in S^1$ and $g\in\Homeo^+(S^1)$ satisfy
\[
f^n(x)=g^n(x)\]
for all $n\in\bZ$, then $\rot(f)=\rot(g)$.
\ee
\end{lem}

The rotation number is a continuous class function (that is, constant on each conjugacy class)
\[
\rot\co \Homeo^+(S^1)\to S^1.\]
Moreover, the rotation number restricts to a group homomorphism on each amenable subgroup of $\Homeo^+(S^1)$; see~\cite{Ghys2001}.

Let us recall the following classical result.

\begin{thm}[H\"older's Theorem~\cite{Holder1901}; see~\cite{Navas2011}]\label{t:hoelder}
A group acting freely on $\bR$ or on $S^1$ by orientation preserving homeomorphisms is abelian.
\end{thm}

We will use the following variation, which is similar to \cite[Theorem 2.2]{FF2001}.
\begin{cor}[cf.~\cite{FF2001}]\label{c:hoelder-s1}
Let  $X$ be a nonempty closed subset of $S^1$,
and let $G$ be a group acting freely on $X$  by orientation preserving homeomorphisms.
Then  the action of $G$ extends to a free action $\rho\co G\to\Homeo^+(S^1)$ such that 
\[\rot\circ\rho\co G\to S^1\] is an injective group homomorphism.
\end{cor}
\bp
If $(a,b)$ is a component of $S^1\setminus X$ and $g \in  G$, then $(ga,gb)$ is also a component of $S^1\setminus X$. So, $G$ extends to some action $\rho$  on $S^1$ in an affine manner.  
Put \[S^1\setminus X = \coprod_{i\ge1} I_i.\] 
If $\rho (g)y = y$ for some $g \in  G$ and for some $y \in  S^1$, then $y \in  I_i$ for some $i$. We have that $ \rho (g)$ restricts to the identity on $ I_i$ by definition. This would imply $\rho (g)\partial I_i = g \partial I_i = \partial I_i$, and so, $g = 1$. That is, the action $\rho$  of $G$ on $S^1$ is free.

By H\"older's theorem, we see $\rho(G)\cong G$ is abelian. Since abelian groups are amenable, 
we have that $\rot \circ \rho$ is a group homomorphism. The freeness of the action $\rho$ implies that $\rot\circ\rho(g)\ne0$ for all nontrivial $g$.\ep

\subsection{Kopell--Denjoy theory}
Let $M\in\{I,S^1\}$.
We denote by $\var(g;M)$ the total variation of a map $g\co M\to\bR$:
\[\var(g;M) = \sup\left\{\sum_{i=0}^{n-1}\left| g(a_{i+1})- g(a_i)\right|
\co (a_i\co 0\le i\le n)\text{ is a partition of }M\right\}.\]
In the case $M=S^1$, we require $a_n=a_0$ in the above definition.
Following \cite{Navas2011},
we say a $C^1$ diffeomorphism $f$ on $M$ is $\Cbv$ if $\var(f';M)<\infty$.
We let $\Diffb(M)$ denote
the group of orientation preserving $\Cbv$ diffeomorphisms of $M$.
The following two results play a fundamental role on the study of $\Cbv$ diffeomorphisms.

\begin{thm}[Denjoy's Theorem {\cite{Denjoy1958}, \cite{Navas2011}}]\label{t:denjoy}
If $a\in\Diffb(S^1)$ and $\Per a=\emptyset$, then $a$ is topologically conjugate to an irrational rotation.
\end{thm}

\begin{thm}[Kopell's Lemma {\cite{Kopell1970}, \cite{Navas2011}}]\label{t:kopell}
Suppose $a\in\Diffb[0,1)$, $b\in\Diff_+^1[0,1)$, and $[a,b]=1$.
If $\Fix a\cap (0,1)=\varnothing$ and $b\ne 1$,
then $\Fix b\cap (0,1)=\varnothing$.
\end{thm}
We remark that the original statement by Kopell was for $C^2$--regularity. Navas extended her result to $C^{1+\mathrm{bv}}$ case~\cite[Theorem 4.1.1]{Navas2011}.

Let $X$ be a topological space. Then we define the \emph{support} of $h\in \Homeo(X)$ as
\[\supp h = X\setminus\Fix h.\] 
It is convenient for us to consider the non--standard \emph{open support} of a homeomorphism as defined here, whereas many other authors use the closure of the open support. We will consistently mean the  open support unless otherwise noted.

For a subgroup $G\le\Homeo(X)$,  we put
\[\supp G = \bigcup_{g\in G}\supp g.\]
We say that $f\in\Homeo(X)$ is \emph{grounded} if $\Fix f\neq\varnothing$. We note that every $f\in\Homeo^+(I)$ is grounded by definition.

The following important observations on commuting $\Cbv$ diffeomorphisms essentially builds on Kopell's Lemma and H\"older's Theorem.

\begin{lem}\label{l:disj-abel}
The following hold:
\be
\item (Disjointness Condition,~\cite{BKK2016})
Let $M\in\{I,S^1\}$,
and let $a,b\in\Diffb(M)$ be commuting grounded diffeomorphisms.
If $A$ and $B$ are components of $\supp a$ and $\supp b$ respectively,
then either $A=B$ or $A\cap B=\varnothing$.
\item (Abelian Criterion, cf.~\cite{FF2001})
If $a,b,c\in\Diffb(I)$ satisfy $\Fix a =\partial I$ and that
$[a,b]=1=[a,c]$, then $[b,c]=1$.
\ee
\end{lem}

\begin{rem}
The abelian criterion as given by Farb and Franks is a straightforward consequence of H\"older's Theorem and Kopell's Lemma. We thank one of the referees for pointing out  another simple proof using Szekeres' Theorem~\cite{Navas2011}.
\end{rem}

The notation $f\restriction_A$ for a function $f$ (or set of functions) and a set $A$ means the restriction to $A$.
We will need the following properties of centralizer groups.

\begin{lem}\label{l:circle}
Let $a\in \Diffb(S^1)$ be an infinite order element,
and let $Z(a)$ be the centralizer of $a$ in $\Diffb(S^1)$.
Then the following hold.
\be
\item\label{p:zz}
If $a$ is grounded and if a group $H\le Z(a)$ is generated by grounded elements, then every element in $H$ is grounded and moreover,
\[\supp a\cap \supp [H,H]=\varnothing.\]
\item\label{p:rot-irr}
If $\rot a\not\in\bQ$, then $Z(a)$ is topologically conjugate to a subgroup of $\operatorname{SO}(2,\bR)$.
\item\label{p:rot-q}
If $\rot a\in\bQ$, then $\rot Z(a)\sse\bQ$.
\item\label{p:hom} 
({cf. \cite[Lemma 3.4]{FF2001}})
The rotation number restricts to a homomorphism on $Z(a)$;
in particular, every element of $[Z(a),Z(a)]$ is grounded.
\ee
\end{lem}
\bp
(\ref{p:zz})
Let $J$ be a component of $\supp a$.
By the Disjointness Condition, the group $H$ acts on the open interval $J$. Since $H$ fixes $\partial J$, every element of $H$ is grounded.
The Abelian Criterion implies that
\[[H,H]\restriction_J=1.\]
So, we have that $J\cap \supp[H,H]=\varnothing$.

(\ref{p:rot-irr}) By Denjoy's Theorem, the map $a$ is topologically conjugate to an irrational rotation. The centralizer of an irrational rotation in $\Homeo^+(S^1)$ is $\operatorname{SO}(2,\bR)$;
see~\cite[Proposition 2.10]{FF2001} or~\cite[Exercise 2.2.12]{Navas2011}.

(\ref{p:rot-q})
Suppose some element $b$ in $Z(a)$ has an irrational rotation number. 
Since $a\in Z(b)$, part (\ref{p:rot-irr}) implies that $a$ is conjugate to a rotation. This is a contradiction, for a rotation with a rational rotation number must have a finite order.

(\ref{p:hom})
If $\rot a$ is irrational, then the conclusion follows from part (\ref{p:rot-irr}). So we may assume $\rot a\in\bQ$. Then $a^p$ is grounded for some $p\ne0$.
Since $Z(a)\le Z(a^p)$, it suffices to prove the lemma for $a^p$. In other words, we may further suppose that $a$ is grounded.

Let us put 
\[G = Z(a),\quad G_0 = \rot^{-1}(0)\cap G.\]
Part (\ref{p:zz}) implies that 
\[
G_0=\bigcup_{x\in S^1}\stab_G(x)=\form{G_0}\]
is a group.
Since $\rot$ is a class function on $\Homeo^+(S^1)$,
we see that $G_0\unlhd G$.

For each $x\in \Fix a$ and $g\in G$, we note
\[
ag(x) = ga(x) = g(x).\]
So, $\Fix a$ is $G$--invariant. 
We have a nonempty proper closed $G$--invariant set
\[
X=\partial \Fix a.\]
\begin{claim}\label{claim:HG}
For all $x\in X$, the group $G_0$ fixes $x$.
\end{claim}
For each $J\in\pi_0\supp a$ and $g\in G_0$, 
we have seen in part (\ref{p:zz}) that $\partial J\sse \Fix g$.
Since we can write
\[
X = \overline{\bigcup\{\partial J\mid J\in\pi_0\supp a\}}\]
we see that $X\sse \Fix g$. This proves the claim.

Let $p\co G\to G/G_0$ denote the quotient map.
By Claim~\ref{claim:HG}, 
 the natural action
\[G/G_0\to \Homeo^+(X),\quad p(g).x=g(x)\] 
is well-defined and free. 
By Corollary~\ref{c:hoelder-s1}, this free action extends to a free action \[\rho\co G/G_0\to \Homeo^+(S^1)\] such that $\rot\circ\rho$ is an injective homomorphism.

For each $g\in G$, $x\in X$ and $n\in\bZ$, we have
\[\rho\circ p(g^n)(x)=g^n(x).\]
By Lemma~\ref{l:inv} (\ref{p:inv}) we see that
\[\rot\circ\rho\circ p = \rot\restriction_G,\]
and hence, that $\rot\restriction_G$ is a group homomorphism.
As $S^1$ is abelian, we also obtain 
\[
[G,G]\le \ker(\rot\restriction_G)=G_0.\qedhere\]
\ep

Note that a finitely generated subgroup of $\operatorname{SO}(2,\bR)$ consisting of elements with rational rotation numbers is necessarily finite. So in Lemma~\ref{l:circle} (\ref{p:rot-q}), if $G$ is a finitely generated subgroup of $Z(a)$ then $\rot(G)$ is a finite subgroup of $S^1\cong\operatorname{SO}(2,\bR)$.

\begin{rem}
Part (\ref{p:hom}) of Lemma~\ref{l:circle} appears in the unpublished work of Farb and Franks \cite{FF2001}, on which our argument is based. We included here a detailed, self-contained proof for readers' convenience. We also remark the necessity of the infinite--order hypothesis, which was omitted in~\cite{FF2001}. For example, let us consider $a,b,c\in\Diff^\infty_+(S^1)$ such that 
for each $x\in S^1=\bR/\bZ$ 
we have
\[
a(x)=x+1/2,\quad
b(x)=x+1/4,\quad
c(x+1/2)=c(x)+1/2\]
and such that
\[
c(0) = 0,\quad
c(1/8) = 1/4,\quad
c(1/4) = 3/8.\]
Then $b,c\in Z(a)$ and $(bc)^3(0)=0$. Hence we have
\[\rot (b)+\rot(c) = 1/4 + 0 \ne 1/3 =\rot(bc).\]
\end{rem}


\begin{lem}\label{l:interval}
Let $a\in \Diffb(I)$,
and let $Z(a)$
be the centralizer of $a$ in $\Diffb(I)$.
Then we have
\[\supp a\cap \supp [Z(a),Z(a)]=\varnothing.\]
\end{lem}
The proof is almost identical to that of Lemma~\ref{l:circle} (\ref{p:zz}).

\subsection{The Two--jumps Lemma}
The Two--jumps Lemma was developed by Baik and the authors in~\cite{BKK2016} and is the second essential analytic result needed to establish Theorem~\ref{thm:main}.

\begin{lem}[Two--jumps Lemma,~\cite{BKK2016}]\label{l:fg}
Let $M\in\{I,S^1\}$ and let $f,g\co M\to M$ be continuous maps.
Suppose  $(s_i), (t_i)$ and $(y_i)$ are infinite sequences of points in $M$
such that for each $i\ge1$, one of the following two conditions hold:
\be[(i)]
\item
$f(y_i)\le s_i = g(s_i) < y_i < t_i = f(t_i) \le g(y_i)$;
\item
$g(y_i)\le t_i = f(t_i) < y_i < s_i = g(s_i) \le f(y_i)$.
\ee
If $|g(y_i)-f(y_i)|$ converges to $0$ as $i$ goes to infinity,
then $f$ or $g$ fails to be $C^1$.
\end{lem}

Figure~\ref{f:fg} illustrates the case (i) of Lemma~\ref{l:fg}.
The reader may note that the homeomorphisms $f$ and $g$ above are \emph{crossed elements}~\cite[Definition 2.2.43]{Navas2011}.
Indeed,  the Two--jumps Lemma generalizes a unpublished lemma of Bonatti--Crovisier--Wilkinson regarding crossed $C^1$--diffeomorphisms,
which can be found in \cite[Proposition 4.2.25]{Navas2011}.

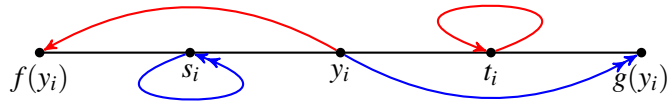
\begin{figure}[h!]
  \tikzstyle {bv}=[black,draw,shape=circle,fill=black,inner sep=1pt]
\begin{tikzpicture}[>=stealth',auto,node distance=3cm, thick]
\draw  (-4,0) node (1) [bv] {} node [below]  {\small $f(y_i)$} 
-- (-2,0)  node  (2) [bv] {} node [below] {\small $s_i$}
-- (0,0) node (3) [bv] {} node [below] {\small $y_i$} 
-- (2,0)  node (4) [bv] {} node [below] {\small $t_i$}
-- (4,0) node (5) [bv] {} node [below]  {\small $g(y_i)$};
\path (3) edge [->,bend right,red] node  {} (1);
\path (3) edge [->>,bend right,blue] node  {} (5);
\draw [->>, blue] (2)  edge [out = 200,in=-20,looseness=50] (2);
\draw [->, red] (4)  edge [out = 20,in=160,looseness=50] (4);
\end{tikzpicture}%
\caption{Two--jumps Lemma.}
\label{f:fg}
\end{figure}


\section{The $C^1$--Smoothability of $\bZ^2\ast\bZ$}
This section establishes the main technical result of the paper as below.

\begin{thm}\label{t:tech-main}
Let $M\in\{I,S^1\}$,
and let
 $a,b,t\in\Diff^1_+(M)$. If 
\[\supp a\cap\supp b=\varnothing,\] 
then the group $\form{a,b,t}$  is not isomorphic to $\bZ^2\ast\bZ$.
\end{thm}

\begin{rem}
\be
\item
We emphasize that this theorem is about $C^1$, rather than $\Cbv$, diffeomorphisms. 
\item The regularity hypothesis of $C^1$ cannot be replaced by $C^0$; see Proposition~\ref{prop:ab-homeo}.
\item The proof of Theorem~\ref{t:tech-main} is relatively easy and standard if  $\supp a$ or $\supp b$ is assumed to have finitely many components.
\ee
\end{rem}

Let us prove Theorem~\ref{t:tech-main} through a sequence of lemmas in this section.

\subsection{Finding a lamplighter group from compact support}
A crucial step in the proof of Theorem~\ref{t:tech-main} is the following construction,
which generalizes a result of Brin and Squier in the PL setting~\cite{BS1985}.
The same idea to find vanishing words from successive commutators goes back even to the Zassenhaus Lemma~\cite{Raghunathan1972};
the authors thank an anonymous referee for suggesting us to further prove the existence of a lamplighter subgroup.

\begin{lem}\label{l:recursive}
Let $1\ne g\in H\le\Homeo^+(I)$. If the closure of $\supp g$ is contained in $\supp H$, then $H$ contains the lamplighter group $\bZ\wr\bZ$.
\end{lem}

More precisely, we will show that for $g_1=g$, there exists a positive integer $m$ and elements $u_1,\ldots,u_m\in H$ such that the recursively defined sequence \[ g_{i+1}=[g_i,u_i g_i u_i^{-1}],\quad i=1,2,\ldots,m,\] satisfies that
$\form{g_m,u_m}\cong\bZ\wr\bZ$
and that $g_{m+1}=1$.
Here and throughout this paper, when $1$ refers to a homeomorphism or a group element then it means the identity, and otherwise it refers to the real number $1$.

\bp[Proof of Lemma~\ref{l:recursive}]
Since $\overline{\supp g_1}$ is a compact subset of the open set $\supp H$, we can enumerate

\[ I_1,I_2,\ldots,I_N\in\pi_0(\supp H)\]
such that $\supp g_1\cap I_i\ne\varnothing$ for each $i$ and such that

 \[\overline{\supp g_1}\sse \bigcup_{i=1}^N I_i.\]
Note that each $I_i$ is an open, $H$--invariant interval contained in $(0,1)$.

Let us inductively construct the elements $u_1,\ldots,u_{k-1}$ satisfying the required properties.
As a base case, we put $r(1)=1\in\bN$ and 
\[ K_1 :=\overline{\supp g_1}\cap I_1.\]
Since $K_1$ is a nonempty compact subset of $I_1\sse\supp H$, we have that
\[
\sup\{ u(\inf K_1)\mid u\in H\}=\sup I_1>\sup K_1.\]
So, there exists $u_1\in H$ such that $K_1\cap u_1^j K_1=\varnothing$ for all $j\in\bN$. Note that 
\[\form{g_1,u_1}\restriction_{I_1}\cong\form{s,t\mid \left[s,t^j st^{-j}\right]=1\text{ for all }j\in\bN}=\bZ\wr\bZ.\]
Let us set \[ r(2) = 1+\sup\{ s\in [1,N]\mid \supp [g_1, u_1^j g_1 u_1^{-j}]\restriction_{I_s}=1\text{ for all }j\in\bN\}\ge 1+r(1)=2.\]
If $r(2)>N$, then we have a sequence of surjections
\[
\bZ\wr\bZ\twoheadrightarrow
\form{g_1,u_1}\twoheadrightarrow
\form{g_1,u_1}\restriction_{I_1}\twoheadrightarrow\bZ\wr\bZ,
\]
which composes to the identity. In particular, $\form{g_1,u_1}\cong\bZ\wr\bZ$.
In the case where $r(2)\le N$, we pick $j\in\bN$ such that the element
\[g_2:= \left[ g_1, u_1^j g_1 u^{-j}\right]\]
satisfies $\supp g_2\cap I_{r(2)}\ne\varnothing$,
and apply the same argument to $g_2$.

By a straightforward induction, we eventually find $1\le m\le r(m)\le N$
and $g_m, u_m\in G$ such that the following hold:
\begin{align*}
&\overline{\supp g_m} \sse I_{r(m)}\cup\cdots\cup I_N,\\
&\supp g_m\cap I_{r(m)}\ne\varnothing,\\
&\supp g_m\cap u_m^j \supp g_m\cap I_{r(m)}=\varnothing,\quad\text{ for all }j\in\bN,\\
&\left[g_m, u_m^j g_m u_m^{-j}\right] =1\quad\text{ for all }j\in\bN.
\end{align*}
It follows that $\form{g_m,u_m}\cong \bZ\wr\bZ$.
\ep


Lemma~\ref{l:recursive} implies the following for circle homeomorphisms.

\begin{lem}\label{l:global}
Let $a,b,c,d\in\Homeo^+(S^1)$ be nontrivial elements such that
\[
\supp a \cap \supp b = \varnothing\quad\text{ and }\quad
\supp c \cap \supp d = \varnothing.\]
If $\supp G= S^1$, then $G$ contains $\bZ\wr\bZ$.
\end{lem}

\bp
For simplicity, let us abbreviate
$\AA  = \pi_0\supp a$, and similarly define 
$\BB$,  $\CC$ and $\DD$.
Since $S^1$ is compact, there exists a finite open covering $\mathcal{V}$ of $S^1$ such that 
\[
\mathcal{V}\sse\AA\cup \BB\cup\CC\cup\DD.\]

By minimizing the cardinality,  we can require that $\mathcal{V}$ forms a \emph{chain of intervals}.
More precisely, this means that
$\mathcal{V}=\{V_1,\ldots,V_k\}$ for some $k\ge1$
and that \[\inf V_i<\sup V_{i-1}\le\inf V_{i+1}\]
for $i=1,2,\ldots,k$, where here the indices are taken cyclically.

Without loss of generality, let us assume $V_1\in\AA $.
Then we have $V_{2i-1}\in\AA \cup \BB$
and $V_{2i}\in\CC\cup\DD$ for each $i$. Note that $k$ is an even number
and that $x=\inf V_2$ is a global fixed point of $H=\form{b,c,d}$.
In particular, we can regard $H$ as acting on $I$, which is a two-point compactification of $S^1\setminus \{x\}$.
Note that
\[
\varnothing\ne \overline{\supp b}\sse S^1
\setminus \bigcup(\mathcal{V}\cap\AA ) \sse \bigcup(\mathcal{V}\cap(\BB\cup\CC\cup\DD))
\sse\supp H.\]
The desired conclusion follows from Lemma~\ref{l:recursive}.
\ep

\subsection{Supports of commutators}
We will need rather technical estimates of supports as given in this subsection.
In order to prevent obfuscation of the ideas, we have included some intuition behind the proofs when appropriate. 
\begin{lem}\label{l:comm-supp}
If $f$ and $g$ are homeomorphisms of a topological space $X$, then
\[\overline{\supp[f,g]}\sse \supp f \cup \supp g \cup \overline{\supp f\cap\supp g}.\]
\end{lem}
\bp
Suppose 
\[x\not\in \supp f\cup \supp g\cup \overline{\supp f\cap \supp g}.\]
Then $f(x)=x=g(x)$. Moreover, for some open neighborhood $U$ of $x$ we have 
\[U\cap \supp f\cap \supp g =\varnothing.\]
We can find an open neighborhood $V\sse U$ of $x$ such that
\[f^{\pm1}(V)\cup g^{\pm1}(V)\sse U.\]
Let $y\in V$. We see $[f,g](y)=y$, by considering the following three cases separately:
\[
y\in V\cap \supp f,\quad y\in V\cap \supp g,\quad y\in V\cap\Fix f\cap \Fix g.\]
So we obtain that
\[[f,g]\restriction_V=1.\]
This implies  \[x\not\in\overline{\supp [f,g]}.\qedhere\]
\ep

\begin{lem}\label{lem:c0} 
Let $X$ be a topological space.
If $b,c,d\in\Homeo(X)$ satisfy \[\supp c\cap\supp d= \varnothing,\]
then for $\phi=[c,bdb^{-1}]$ we have that
\[
\supp\phi \sse \supp b \cup cb(\supp b \cap\supp d)\cup db^{-1}(\supp b \cap \supp c).\]
\end{lem}

Let us briefly explain the key idea behind the statement of this lemma. The support of the homeomorphism $bdb^{-1}$ is exactly $b(\supp d)$. The homeomorphism $\phi$ may be viewed as a composition of $bd^{-1}b^{-1}$ and the conjugate of $bdb^{-1}$ by $c$, and the latter of these has support $cb(\supp d)$. 
By exhaustively checking the possible images of points under $bdb^{-1}$ and $cbdb^{-1}c^{-1}$, we see that every $x\in\supp \phi$ belongs to one of the three sets as stated in the lemma.

\bp[Proof of Lemma~\ref{lem:c0}]
For brevity, let us write 
\[
\tilde b = \supp b,\quad
\tilde c = \supp c,\quad
\tilde d = \supp d.\]
Let us consider three equivalent expressions for $\phi$:
\[
 [c,bdb^{-1}]=cb d (cb)^{-1}\cdot bd^{-1}b^{-1}= c\cdot b\cdot db^{-1}c^{-1}(db^{-1})^{-1}\cdot b^{-1}.\]
After some set theoretic computation, one sees the following.
\begin{align*}\supp  \phi&\sse\left(\tilde c\cup b\tilde d\right)\cap \left(cb \tilde d\cup b\tilde d\right)\cap\left(\tilde c\cup \tilde b\cup db^{-1}\tilde c\right)\\&\sse\left((\tilde c\cap cb \tilde d)\cup b\tilde d\right)\cap\left(\tilde c\cup \tilde b\cup db^{-1}\tilde c\right)\\&\sse(\tilde c\cap cb \tilde d) \cup \left( b\tilde d\cap (\tilde c\cup \tilde b\cup db^{-1}\tilde c)\right)\\&\sse\left(\tilde c\cap cb \tilde d\right) \cup \left( (\tilde b\cup \tilde d)\cap (\tilde b\cup\tilde c\cup db^{-1}\tilde c) \right)\\ &\sse\left(\tilde c\cap cb \tilde d\right)\cup\tilde b\cup \left(\tilde d\cap (\tilde c\cup db^{-1}\tilde c)\right)\\ &\sse\left(\tilde c\cap cb \tilde d\right)\cup\tilde b \cup\left(\tilde d\cap db^{-1}\tilde c\right). \end{align*}
Note that we used $b\tilde d\sse \tilde b\cup \tilde d$, and also $\tilde c\cap\tilde d=\varnothing$.
It now suffices for us to prove the following claim:
\begin{claim*}\label{claim:cbbd}
We have the following:
\begin{align*}
\tilde c\cap cb\tilde d&\sse cb\left(\tilde b\cap\tilde d\right),\\
\tilde d\cap db^{-1}\tilde c&\sse db^{-1}\left(\tilde b\cap\tilde c\right).
\end{align*}
\end{claim*}
To see the first part of the claim, let us consider 
$x\in X$ satisfying
\[cb(x)\in \tilde c\cap cb \tilde d.\]
Then we have $x\in \tilde d$ and $cb(x)\in \tilde c$.
Since
 $\tilde c\cap \tilde d=\varnothing$
and $b(x)\in c^{-1}\tilde c=\tilde c$, we see $x\ne b(x)$.
In particular, we have $x\in\tilde b$
 and $cb(x)\in cb(\tilde b\cap \tilde d)$. This proves the first part of the claim. The second part follows by symmetry.\ep
 
\begin{lem}\label{lem:c1} 
If $b,c,d\in\Diff^1_+(I)$ are given such that \[\supp c\cap\supp d= \varnothing,\]
then for $\phi=[c,bdb^{-1}]$ we have that
\[\overline{\supp\phi\setminus \supp b}\sse\supp c\cup\supp d.\]
\end{lem}
\bp
As in the proof of Lemma~\ref{l:global}, we let $\BB=\pi_0\supp b$ and $\CC=\pi_0\supp c$.
Let
\[J_B=B\cup cb(B\cap\supp d)\cup db^{-1}(B\cap \supp c)\]
for each $B\in\BB$.
By Lemma~\ref{lem:c0}, we have that
\[\supp\phi\sse \bigcup\{J_B\mid B\in\BB\}
=  \bigcup\{J_B\setminus B\mid B\in \BB\}\cup \supp b.\]
Moreover, for each $B\in\BB$ we note that
\[
\overline{
J_B\setminus B}
\sse \overline{c(B)\setminus B}\cup \overline{d(B)\setminus B}
\sse\supp c \cup\supp d.\]

\begin{claim*}\label{claim:jbb}
The following set is a finite collection of intervals:
\[ \BB_0=\{B\in\BB\mid J_B\ne B\}.\]
\end{claim*}
We will employ the $C^1$--hypothesis for this claim. 
Let us write
\[\BB_1=\{B\in\BB\mid cb(B\cap \supp d)\setminus B\ne\varnothing\},\ 
\BB_2=\{B\in\BB\mid db^{-1}(B\cap \supp c)\setminus B\ne\varnothing\}.\]
Assume for a contradiction that $\BB_0=\BB_1\cup\BB_2$ is infinite. 
We may suppose $\BB_1$ is infinite, as the proof is similar when $\BB_2$ is infinite.
There are infinitely many distinct $B_1,B_2,\ldots\in \BB_1$ 
and $x_i\in B_i\cap\supp d$ such that $cb (x_i)\not\in B_i$. 
Then we have $C_1,C_2,\ldots\in \CC$ such that
$b (x_i), cb(x_i)\in C_i$.
Since $x_i\in\supp d$, we have $x_i\not\in C_i$; see Figure~\ref{f:JB}.

Let us consider the interval $J_i = [x_i,cb (x_i)]$ which contains $b(x_i)$ in the interior, up to switching the endpoints of this interval.
Then we have
\[(b(x_i),cb (x_i)]\cap \partial B_i\ne\varnothing,\quad [x_i,b (x_i))\cap\partial C_i\ne\varnothing.\]
We now apply the Two--jumps Lemma (Lemma~\ref{l:fg}) to the following parameters
\[f=b^{-1},\ g=c,\ s_i=\partial C_i\cap B_i,\  t_i=\partial B_i\cap C_i,\ y_i=b(x_i).\]
We deduce that $b$ or $c$ is not $C^1$.
This is a contradiction and the claim is proved.

From the claim above, we deduce the conclusion as follows.
\[
\overline{\supp\phi\setminus\supp b}
\sse\overline{\bigcup\{J_B\setminus B\mid B\in\BB_0\}}=\bigcup\{\overline{J_B\setminus B}\mid B\in\BB_0\}
\sse\supp c\cup\supp d.\qedhere\]
\ep

\begin{figure}[h!]
  \tikzstyle {a}=[black,postaction=decorate,decoration={%
    markings,%
    mark=at position 1 with {\arrow[black]{stealth};}    }]
  \tikzstyle {bv}=[black,draw,shape=circle,fill=black,inner sep=1.5pt]
{
\begin{tikzpicture}[thick,scale=.7]
\draw [red,ultra thick] (-4,0.5) -- (2,0.5);
\draw [blue,ultra thick] (0,0) -- (6,0);
\draw [teal,ultra thick] (4,.5) -- (10,.5);
\draw (-3,0) node [] {\small $\supp d$}; 
\draw (9,0) node [] {\small $C_i$}; 
\draw (3,-.5) node [] {\small $B_i$}; 
\draw [dashed] (1,1.4) -- (1,-.3) node [below]  {\small $x_i$};
\draw [dashed] (5,1.4) -- (5,-.3) node [below]  {\small $b(x_i)$};
\draw [dashed] (7.5,1.4) -- (7.5,-.3) node [below]  {\small $cb(x_i)$};
\draw (1,1) -- (7.5,1);
\draw (3,1.4) node [] {\small $J_i$}; 
\end{tikzpicture}%
}
\caption{Lemma~\ref{lem:c1}.}
\label{f:JB}
\end{figure}

\subsection{Finding compact supports}
We will deduce Theorem~\ref{t:tech-main} from the following, seemingly weaker result.
\begin{lem}\label{l:main}
Let $M\in\{I,S^1\}$ and let
$a,b,c,d\in\Diff^1_+(M)$.
If 
\[
\supp a\cap\supp b=\varnothing,\quad
\supp c\cap\supp d=\varnothing,\]
then the group $\form{a,b,c,d}$ is not isomorphic to $\bZ^2\ast\bZ^2$.
\end{lem}

Let us note two properties of RAAGs. First, a RAAG does not contain a subgroup isomorphic to $\bZ\wr\bZ$. The reason is that, every two--generator subgroup of a RAAG is either free or free abelian~\cite{Baudisch1981}; see also~\cite[Corollary 1.3]{KK2015GT}.
Second, a RAAG is \emph{Hopfian}; that is, every endomorphsm of a RAAG is an isomorphism. This follows from a general fact that every finitely generated residually finite group is Hopfian~\cite{LS2001}.

\bp[Proof of Theorem~\ref{t:tech-main} from Lemma~\ref{l:main}]
Assume $\form{a,b,t}\cong\bZ^2\ast\bZ$. 
Since the RAAG $\bZ^2\ast\bZ$ is Hopfian, the natural surjection between groups
\[ \form{A,B,T\mid [A,B]=1}\to \form{a,b,t}\]
is actually an isomorphism. It follows that
\[\form{a,b,tat^{-1},tbt^{-1}}\cong\form{A,B,TAT^{-1},TBT^{-1}}\cong\bZ^2\ast\bZ^2.\]
This contradicts Lemma~\ref{l:main}, since the four diffeomorphism $a,b,tat^{-1},tbt^{-1}$ satisfy the conditions of the lemma.
\ep

\bp[Proof of Lemma~\ref{l:main}]
We put $G=\form{a,b,c,d}$ and  consider an abstract group
\[ G_0 = \form{a_0,b_0,c_0,d_0\mid [a_0,b_0]=1=[c_0,d_0]}\cong \bZ^2\ast\bZ^2.\]
There is a natural surjection $p\co G_0\to G$ defined by \[(a_0,b_0,c_0,d_0)\mapsto(a,b,c,d).\] 

Assume for a contradiction that $G\cong G_0$. By the Hopficity of $\bZ^2\ast\bZ^2$,  we see that $p$ is an isomorphism. Since $G$ does not contain $\bZ\wr\bZ$,  Lemma~\ref{l:global} implies that $G$ has a global fixed point.
In other words, we may assume $M=I$.

Let us define 
$\phi=[c,bdb^{-1}]$ and $\psi=[\phi,a]$.
Lemma~\ref{l:comm-supp} implies that
\[\overline{\supp\psi}\sse\supp\phi\cup\supp a\cup\overline{\supp\phi\cap\supp a}.\]
We see from Lemma~\ref{lem:c1} that
\[\overline{\supp\phi\cap\supp a}
\sse\overline{\supp\phi\setminus\supp b}
\sse\supp c\cup \supp d.\]
So, it follows that
\[
\overline{\supp \psi}\sse\supp G.\]

As we are assuming $p$ is injective, we have 
\[\psi = [\phi,a]=\left[ [c, bdb^{-1}],a\right]\ne1.\]
Lemma~\ref{l:recursive} implies that $G$ contains $\bZ\wr\bZ$, which is a contradiction. This completes the proof.\ep

Let us conclude this section by describing one generalization of Theorem~\ref{t:tech-main}.

\begin{prop}\label{t:tech-main2}
Let $M\in\{I,S^1\}$.
If $a,b,c,d\in\Diff^1_+(M)$ satisfy 
\[\supp a\cap\supp b=\varnothing,\quad
\supp c\cap\supp d=\varnothing.\]
and 
\[
\left[[c,bdb^{-1}],a\right]\ne1,
\]
then $\form{a,b,c,d}$ contains the lamplighter group $\bZ\wr\bZ$.
\end{prop}

In particular, the group $\form{a,b,c,d}$ does not embed into a RAAG.

\section{Proof of Theorem~\ref{thm:main}}\label{s:main-thm}

In this section, we apply the facts we have gathered to complete the proof of the main result.

\subsection{Reducing to the connected case}
We will reduce the proof of Theorem~\ref{thm:main} to the case $M\in\{I,S^1\}$, using the following group theoretic observations.

\begin{lem}\label{lem:L1}
Suppose $A,B,C,D$ are groups, and suppose that $A\times B$ is a normal subgroup of $C*D$. Then at least one of these four groups is trivial.
\end{lem}
\begin{proof}
If $A\times B$ is a subgroup of $C*D$ then the Kurosh Subgroup Theorem implies that there is a free product decomposition \[A\times B\cong F\ast\bigast_{i} H_i,\] where $F$ is a free group (possibly of infinite rank) and where each $H_i$ is conjugate into $C$ or into $D$. By analyzing centralizers of elements, it is easy to show that a nontrivial free product is never isomorphic to a nontrivial direct product (cf.~\cite[p.177]{LS2001}). It follows that $A\times B$ is conjugate into $C$ or $D$, which contradicts the normality of $A\times B$.
\end{proof}

An alternative proof of Lemma~\ref{lem:L1} can be given using Bass--Serre theory (see~\cite{Serre1977}).

\begin{lem}\label{lem:L2}
Suppose $A,B,C,D$ are nontrivial groups, and that $A*B\le C\times D$. Then there is an injective homomorphism from $A*B$ into either $C$ or $D$.
\end{lem}
\begin{proof}
Suppose the contrary, so that $K_C$ and $K_D$ are the (nontrivial) kernels of the inclusion of $A*B$ into $C\times D$ composed with the projections onto $C$ and $D$. Then $K_C\cap K_D=1$ and $K_C$ and $K_D$ normalize each other, so that $K_CK_D\cong K_C\times K_D\le A*B$. This contradicts Lemma~\ref{lem:L1}.
\end{proof}

\begin{lem}\label{lem:connected}
Suppose \[M=\coprod_{i=1}^n M_i\] is a compact one--manifold, and suppose that $A\ast B$ embeds into $\Diffb(M)$. Then 
 for some finite index subgroups $A_0\le A$ and $B_0\le B$,
  and for some $i$, we have an embedding of $A_0\ast B_0$ into $\Diffb(M_i)$.
\end{lem}
\begin{proof}
This follows immediately from Lemma~\ref{lem:L2}, using the fact that $\Diffb(M)$ is commensurable with \[\prod_{i=1}^n\Diffb(M_i).\]

Note that passage to  finite index subgroups is necessary, since $M$ may consist of a union of diffeomorphic manifolds which are permuted by the action of $A\ast B$.
\end{proof}

\subsection{Taming supports}

Let us denote the center of a group $G$ by $Z_G$.
If $M$ is a one--manifold and if $s\in\Diffb(M)$, then
we denote by $Z(s)$ the centralizer of $s$ in $\Diffb(M)$.
The following lemma is crucial for applying Lemma~\ref{l:main}:

\begin{lem}\label{lem:Z2}
Assume one of the following:
\be[(i)]
\item
$M=I$ and $G$ is a nonabelian group such that $Z_G\ne1$.
\item
$M=S^1$ and $G$ is a non-metabelian group such that $\bZ\le Z_G$.
\item
$M=S^1$ and $G$ is a finitely generated group such that $\bZ\le Z_G$
and such that $G$ is not abelian-by-finite cyclic.
\ee
In each of the cases, if $G\le\Diffb(M)$, then there is a subgroup $\bZ^2\le G$ generated by diffeomorphisms $a$ and $b$ such that $\supp a\cap\supp b=\varnothing$.
\end{lem}

\bp

\emph{Case (i).} Let us pick $s\in Z_G\setminus1$ and $b\in[G,G]\setminus1$. 
Since $G\le Z(s)$,  Lemma~\ref{l:interval} implies that 
\[\supp s\cap \supp b=\varnothing.\]
Since $\Homeo^+(I)$ is torsion-free, we have $\form{b,s}\cong\bZ^2$ as desired.

\emph{Case (ii).}
We are given with some $s\in Z_G$ such that $\form{s}\cong\bZ$. As $G$ is nonabelian, Lemma~\ref{l:circle} (\ref{p:rot-irr}) implies that $\rot(s)\in\bQ$; in particular, $s^n$ is grounded for some $n\ge1$.
From part (\ref{p:hom}) of the same lemma
and from that $G\le Z(s)$, we see 
every element of $[G,G]$ is grounded.
Since $[G,G]\le Z(s^n)$, we note from Lemma~\ref{l:circle} (\ref{p:zz}) that 
\[\supp(s^n)\cap \supp G''=\varnothing.\]
From the metabelian hypothesis, we can find $b\in G''\setminus1$.
As $s^n$ and $b$ are grounded, they have infinite orders. It follows that $\form{s^n,b}\cong\bZ^2$.

\emph{Case (iii).}
Let us proceed similarly to the case (ii). Namely, pick $s\in Z_G$ such that $\form{s}\cong\bZ$. 
By Lemma~\ref{l:circle}, we have a homomorphism
\[
\rot\restriction_G\co G\to \bQ.\]
We fix $n\ge1$ such that $s^n$ is grounded.
As $G$ is finitely generated, we see $\rot(G)$ is finite cyclic.
The hypothesis implies that $G_0=\ker(\rot\restriction_G)$ is not abelian.
Since every element of $G_0$ is grounded, we can apply
 Lemma~\ref{l:circle} (\ref{p:zz}) and deduce 
\[\supp s^n\cap \supp [G_0,G_0]=\varnothing.\]
Each $b\in [G_0,G_0]\setminus1$ then yields the desired subgroup
$\form{b,s^n}\cong\bZ^2$.
\ep

\subsection{The main result}


\begin{proof}[Proof of Theorem~\ref{thm:main}]
Suppose $(G\times\Z)*\Z\le \Diffb(M)$ for some compact one--manifold $M$. Replacing $G$ by a finite index subgroup if necessary, we may assume that $M$ is connected, by Lemma~\ref{lem:connected}. 
By applying the cases (i) and (ii) of Lemma~\ref{lem:Z2} to the group $G\times \bZ$,
we can find a subgroup $\langle a,b\rangle\cong\bZ^2\le G\times\Z$ such that $\supp a\cap\supp b=\varnothing$.
If we write the $\bZ$--free factor of $(G\times \bZ)\ast\bZ$ as $\form{t}$, then
\[\form{a,b,t}\cong\form{a,b}\ast\form{t}\cong\bZ^2\ast\bZ.\]
This contradicts Theorem~\ref{t:tech-main}.
\end{proof}

One can now deduce Corollary~\ref{cor:fg1} as well as Corollary~\ref{cor:cnt} below from Lemma~\ref{lem:Z2}, in the exact same fashion as Theorem~\ref{thm:main}. 
\begin{cor}\label{cor:cnt}
Let $G$ be a group.
\be
\item 
If $G$ is nonabelian and if the center of $G$ is nontrivial,
then
$G*\Z$ admits no faithful $\Cb$ action on $I$. 
\item Suppose $G$ is finitely generated.
If $G$ is not abelian-by-finite cyclic
and if the center of $G$ contains a copy of $\bZ$,
then 
$G*\Z$ admits no faithful $\Cb$ action on $S^1$. 
\end{enumerate}
\end{cor}

Here, a group $G$ is $\mathcal{X}$--by--$\mathcal{Y}$ for group theoretic properties $\mathcal{X}$ and $\mathcal{Y}$ if there is an exact sequence \[1\to K\to G\to Q\to 1\] such that $K$ has property $\mathcal{X}$ and $Q$ has property $\mathcal{Y}$. We allow both $K$ and $Q$ to be trivial.

Note that $G\times \bZ$ often occurs as a subgroup of $\Diffb(M)$, where $G$ is not virtually metabelian. We have the following immediate consequence:

\begin{cor}\label{cor:not closed}
Let $\mathcal{G}$ denote the class of finitely generated subgroups of $\Diffb(M)$. The class $\mathcal{G}$ is not closed under taking finite free products.
\end{cor}

\section{Smooth right-angled Artin group actions on compact one--manifolds}

In this and the remaining sections, we deduce several corollaries from Theorem \ref{thm:main}. We first complete the classification of right-angled Artin groups which admit faithful $C^{\infty}$ actions on a compact one--manifold (Corollary \ref{cor:classification}).

\begin{lem}\label{lem:dichotomy}
Let $A(\gam)$ be a right-angled Artin group. Then one of the following mutually exclusive conclusions holds:
\begin{enumerate}
\item
We have $(F_2\times\Z)*\Z\le A(\gam)$;
\item
The graph $\gam$ lies in $\mathcal{K}_3$.
\end{enumerate}
\end{lem}
\begin{proof}
Let us consider a stratification of graph classes:
\[ \KK_2\sse\KK_3\sse\KK.\]

Suppose $\gam\in\KK_2$. Then $A(\gam)$ is the free product of free abelian groups, and hence contains no copy of $(F_2\times\Z)*\Z$.

Let $\gam\in\KK_3\setminus\KK_2$. Then $\gam$ is the join of at least two graphs $\gam_1,\gam_2$ in $\KK_2$. 
We write
\[
A(\gam)=A(\gam_1)\times A(\gam_2).\]
If $A(\gam)$ contains a copy of $(F_2\times\bZ)\ast\bZ$,
then so does $A(\gam_1)$ or $A(\gam_2)$
by Lemma~\ref{lem:L2}; this would contradict the previous paragraph.

Assume $\gam\in\KK\setminus\KK_3$.
First consider the case that $\gam\in\KK_{2i}\setminus\KK_{2i-1}$ for some $i\ge2$. 
We can write
\[
\gam=\coprod_{j=1}^k \gam_j\]
for some $k\ge2$ and for some nonempty connected graphs $\gam_j\in\KK_{2i-1}$. 
These graphs $\Gamma_j$ cannot all be complete graphs, for otherwise $\gam\in\KK_2$. 
So at least one graph $\Gamma_j$ contains $P_3$, the path on three vertices, as a full subgraph. 
This implies that $A(\gam)$ contains a copy of $(F_2\times\bZ)\ast\bZ$.

We then consider the case that $\gam\in\KK_{2i+1}\setminus\KK_{2i}$ for some $i\ge2$. 
Note $\gam$ is the join of some graphs $\gam_1,\ldots,\gam_k$ in $\KK_{2i}$. 
By the previous graph, each $A(\gam_j)$ contains $(F_2\times\bZ)\ast\bZ$.

Finally assume $\gam\notin\KK$, so that $\gam$ is not a cograph. 
Then we have that $P_4$ is a full subgraph of $\gam$, so that $A(P_4)\le A(\gam)$. The group $A(P_4)$ contains every right-angled Artin group $A(F)$, where $F$ is a finite forest (see~\cite{KK2013}). Since the defining graph of $(F_2\times\Z)*\Z$ is a copy of a path $P_3$ on three vertices together with an isolated vertex, its defining graph is a finite forest. We see that $(F_2\times\Z)*\Z\le A(\gam)$.
\end{proof}

We complete the proof of Corollary \ref{cor:classification} with the following proposition:

\begin{prop}\label{prop:cinfty}
Let $\gam\in\mathcal{K}_3$ and let $M$ be a compact one--manifold. Then there is an embedding of $A(\gam)$ into $\Diff_+^{\infty}(M)$.
\end{prop}

In the case when $M=S^1$, we will prove more precise facts in Section \ref{sec:semiconj}.

\begin{proof}[Proof of Proposition~\ref{prop:cinfty}]
Let $\Diff^{\infty}_0(I)$ denote the group of $C^{\infty}$ diffeomorphisms of the interval which are infinitely tangent to the identity at $\{0,1\}$. It suffices to find a copy $A(\gam)\le\Diff^{\infty}_0(I)$, since $I$ is a submanifold of every compact one--manifold $M$ and every element of $\Diff^{\infty}_0(I)$ can, by definition, be extended to all of $M$ by the identity map. 

In the case when $\gam\in\mathcal{K}_2$, we can write 
$A(\gam)=\Z^{n_1}\ast\cdots\ast\Z^{n_k}$.
As we have an embedding $A(\gam)\hookrightarrow\Z*\Z^N$ for $N=\max_i n_i$, it suffices to prove the proposition for $\Z*\Z^N$ in this case. 
To do this, we first find a copy of $\Z^N\le \Diff^{\infty}_0(I)$ such that the support of each nontrivial element of $\Z^N$ is all of $(0,1)$. The existence of such a copy of $\Z^N$ follows from choosing a $C^{\infty}$ vector field on $I$ which vanishes only at $\partial I$ and integrating it to get a flow, which gives an $\R$--worth of commuting elements of $\Diff^{\infty}_0(I)$.
Now, choosing a generic (in the sense of Baire) element $\psi$ of $\Diff^{\infty}_0(I)$, we have that $\psi$ and this copy of $\Z^N$ generate a copy of $\Z*\Z^N\le\Diff^{\infty}_0(I)$ (cf.~\cite{KKMj2016}).



Finally, if $\gam\in\mathcal{K}_3\setminus\mathcal{K}_2$ then $A(\gam)$ is a finite direct product of $k$ right-angled Artin groups with defining graphs in $\mathcal{K}_2$. Write again $N$ for the maximal rank of an abelian subgroup of $A(\gam)$. We choose a finite collection of disjoint intervals $\{J_1,\ldots,J_k\}$ with nonempty interior inside of $I$, and realize a copy of $\Z*\Z^N$ on each $J_i$, extending by the identity outside of $J_i$. It is clear that $A(\gam)$ is thus realized as a subgroup of $\Diff^{\infty}_0(I)$.
\end{proof}

\section{Lower regularity}

In this section, we prove Proposition \ref{prop:free prod}. 
Recall a \emph{left order} on a group $G$ is a total order $\prec$ on $G$ such that for all triples $a,b,g\in G$ we have $a\prec b$ if and only if $ga\prec gb$. 
A group is \emph{left orderable} if it admits a left order.

Every subgroup of $\Homeo^+(\bR)$ is left orderable. Conversely, if $G$ is countable and left orderable,  then there is a faithful action $G\to\Homeo^+(\bR)$;
it can be further required from the action that for some fixed point $x_0\in\bR$,
whenever $g\prec h$ we have $g(x_0)<h(x_0)$. There exists a standard example of such an action, called a \emph{dynamical realization} of the given left order~\cite{Navas2011}.


Note that if $\prec$ is a left order on a group $G$, then $G$ also admits the opposite order $\prec^{opp}$, where $g\prec h$ if and only if $h\prec^{opp} g$.  It follows that if $g$ is a nontrivial element of a countable left orderable group $G$, then there exists a faithful action of $G$ on $\R$ such that $x_0<g(x_0)$ for some $x_0\in \R$. We will call an action coming from the opposite order as an \emph{opposite action}.

Let us now establish Proposition \ref{prop:free prod} for the case $M=I$:

\begin{prop}\label{prop:homeo closed}
Let $\mathcal{G}$ denote the class of countable subgroups of $\Homeo^+(I)$. Then $\mathcal{G}$ is closed under countable free products.
\end{prop}

Proposition \ref{prop:homeo closed} also follows from the general fact that if $G$ and $H$ are countable left orderable groups then so is $G\ast H$, as was shown in~\cite{KM1996,Rivas2012JAlg}. 
We are including a proof as the idea will be needed for Proposition~\ref{prop:ab-homeo}.
Unlike a dynamical realization of $G\ast H$ coming from its natural left order~\cite{DNR2014}, the construction below may not have a point $x_0$ with trivial stabilizer.

\bp
The proof is based on the idea of~\cite{BKK2014}. 
By induction, it suffices to show that if $G_1,G_2,\ldots\in\mathcal{G}$ then $\bigast_{i\ge1}G_i\in\mathcal{G}$. Since $K=\langle \{G_i\mid i\ge1\}\rangle\in\mathcal{G}$ and since $\bigast_{i\ge1}K\le \Z*K$ 
by the normal form theorem for free products,
it suffices to show that if $G\in\mathcal{G}$ then $\Z*G\in\mathcal{G}$.
We let $H=G\ast \bZ$, and choose countably many disjoint subintervals $\{I_h\}_{1\neq h\in H}$ of $I$, each with nonempty interior. 

Let $h\in H\setminus1$ be arbitrary.
We claim that there exists an action $\rho_h\co H\to\Homeo^+(I_h)$ such that $h\not\in\ker\rho_h$.
For some $g_i\in G$ and $r_i\in\bZ$, we can write
\[ h =g_\ell t^{r_\ell} \cdots g_1t^{r_1} .\]
If $h\in\form{t}$ or $h\in G$, then the claim is obvious.
So, possibly after conjugation we may assume that $r_i\ne0$ and $g_i\ne1$ for each $i$. 

We choose disjoint closed intervals \[\{J_1,\ldots,J_\ell\}\subset I_h\] with nonempty interior, $J_i=[p_i,q_i]$ and $q_i<p_j$ for $1\leq i<j\leq \ell$.
For each $1\leq i\leq \ell$, we choose an action of $G$ on $J_i$ such that
 $x_i<g_i(x_i)$ for some $x_i\in J_i$.

We now define an action of $t$ on $I_h$. We choose $\ell$ disjoint closed intervals \[\{L_1,\ldots,L_\ell\}\subset I_h\] of the form $L_i=[c_i,d_i]$. We choose $c_1<x_0<p_1$ and $x_1<d_1<g_1(x_1)$. For $i>1$, we choose points $d_{i-1}<c_i<g_{i-1}(x_{i-1})$ and $x_i<d_i<g_i(x_i)$. We now define $t$ on \[\bigcup_{i=1}^\ell L_i\] so that $t^{r_1}(x_0)=x_1$ and so that $t^{r_i}(g_{i-1}(x_{i-1}))=x_i$. Since the $\{L_1,\ldots,L_\ell\}$ are disjoint, such a choice for the definition of $t$ is possible. It is routine to check that $x_0<h(x_0)=g_\ell(x_\ell)$. 
Since $h$ is not in the kernel of this action, the claim is proved.

By taking disjoint subintervals $\{I_h\}_{h\in H\setminus 1}$ of $\R$ with each $I_h$ equipped with an action of $H$ by $\rho_h$, we obtain the desired embedding
\[\rho = \prod_{h\in H\setminus1}\rho_h \co H\to\prod_{h\in H\setminus1}\Homeo^+(I_h)\le  \Homeo^+(\bR)\cong\Homeo^+(I).\qedhere\]
\ep

We note from the above proof that the restriction of $G$ on each $J_i=[p_i,q_i]$ corresponds to the original action of $G$, or the opposite action of $G$. 
By taking $G=\bZ^2$, we obtain the following.

\begin{prop}\label{prop:ab-homeo}
There exists an embedding
\[ \rho\co \form{a,b,t\mid [a,b]=1}\cong\bZ^2\ast\bZ\to\Homeo^+(I)\]
such that $\supp\rho(a)$ and $\supp\rho(b)$ are disjoint.
\end{prop}
In particular, the $C^1$ hypothesis in Theorem~\ref{t:tech-main} cannot be lowered to $C^0$.

\section{Complexity of right-angled Artin groups versus diversity of circle actions}\label{sec:semiconj}

In this section, we prove Corollary~\ref{cor:conj}. The proof follows easily from the results in~\cite{KKMj2016}, together with Corollary~\ref{cor:classification}.

In a joint work with Mj~\cite{KKMj2016}, the authors defined a class of finitely generated groups $\mathcal{F}$ and called each group in the class as \emph{liftable--flexible}.
Let us extract the necessary facts which are demonstrated therein.
Recall from the introduction that a \emph{projective} action of a group is a representation into $\PSL(2,\bR)$.

\begin{thm}[cf. Theorem 1.1 of~\cite{KKMj2016}]\label{thm:KKMj}
There exists a class of finitely generated groups $\FF$ satisfying the following.
\be
\item For each $G\in\FF$, there exist uncountably many distinct semiconjugacy classes of faithful projective actions on $S^1$.
\item
Every limit group is in $\FF$; in particular, every free group and every free abelian group is in $\FF$.
\item
If $G,H\in\FF$, then $G*H\in\FF$.
\end{enumerate}
\end{thm}

\begin{lem}\label{l:ab}
Let $A$ and $B$ be nontrivial torsion-free finitely generated subgroups such that $B$ is nonabelian. 
We assume that $A\times B\le \Diffb(S^1)$.
Then there exist finite index normal subgroups $A_1\unlhd A$ and $B_1\unlhd B$ such that 
 $A_1\times B_1$ has a global fixed point on $S^1$.
Furthermore, the closure of $\supp A_1$ is a proper subset of $S^1$.
\end{lem}
\bp
Choose an arbitrary element $a\in A\setminus 1$.
Since the centralizer $Z(a)$ is nonabelian and $A$ is torsion-free, 
Lemma~\ref{l:circle} implies that $\rot Z(a)\sse\bQ$.
Since $B$ and $a\times B$ are subsets of $Z(a)$, we see that $\rot(B), \rot(a\times B)\sse\bQ$.
This shows that $\rot(A\times B)\sse\bQ$.
 
For each $a\in A$ and $b\in B$, note that 
 \[a,b,ab\in Z(a).\]
By Lemma~\ref{l:circle} (\ref{p:hom}), it follows that $\rot(a)+\rot(b)=\rot(ab)$.
In particular, we have a homomorphism 
\[\rot\co A\times B\to \bQ.\]
Since $A\times B$ is finitely generated, the image of the above map is finite.
So we can find finite index normal subgroups $A_1\unlhd A$ and $B_1\unlhd B$ such that every element of $A_1\times B_1$ is grounded. 

Let $a\in A_1\setminus1$.
Since $B_1$ centralizes $a$, it acts on each component of $\supp a$ as an abelian group; see Lemma~\ref{l:disj-abel}. It follows that $B_1$ preserves the set $\supp A_1$ and
\[\supp A_1\cap\supp [B_1,B_1]=\varnothing.\]
Since $B_1$ is nonabelian, we see $\overline{\supp A_1}\ne S^1$. 
It follows that $\partial\supp A_1$ is a global fixed point of $A_1\times B_1$.
\ep

\begin{proof}[Proof of Corollary~\ref{cor:conj}]

(1) 
If $\gam\in\mathcal{K}_2$, then $A(\gam)$ is a finite free product of free abelian groups. Combining parts (2) and (3) of Theorem~\ref{thm:KKMj}, we see $A(\gam)$ is in the class $\FF$. The  first part of the same theorem implies the desired conclusion.

(2) Suppose that $\gam\in\mathcal{K}_3\setminus\mathcal{K}_2$, so that $A(\gam)$ decomposes as a nontrivial direct product $A(\gam_1)\times A(\gam_2)$, where at least one of $A(\gam_1)$ and $A(\gam_2)$ is nonabelian. Say $A(\gam_2)$ is nonabelian. 
By Lemma~\ref{l:ab}, there exist finite index normal subgroups $G\unlhd A(\gam_1)$
and $H\unlhd A(\gam_2)$ 
such that $G\times H$ has a global fixed point; in particular, $A(\gam)$ has a finite orbit. 
We put
\[X = \overline{\supp G}, \quad Y = S^1\setminus X.\]
Since $G$ is normal in $A(\gam)$ we see that 
$X$ and $Y$ are $A(\gam)$--invariant sets of $S^1$, which are proper by the same lemma.
This implies that $X$ and $Y$ have nonempty interiors, and hence $A(\gam)$ does not have a dense orbit.

(3) This follows immediately from Corollary~\ref{cor:classification}.
\end{proof}

We briefly remark that if a group acts on $S^1$ with a global fixed point then it is semi--conjugate to a trivial action, and if it acts with a periodic point then the action is semi--conjugate to a rational rotation group. Corollary~\ref{cor:conj} implies that if $\gam\in\mathcal{K}_3\setminus\mathcal{K}_2$ then $A(\gam)$ admits only countably many semi--conjugacy classes of faithful actions, and it is not difficult to realize one semi--conjugacy class for each rational rotation.

Observe that since every right-angled Artin group surjects to $\Z$, every right-angled Artin group admits uncountably many distinct semi--conjugacy classes of non--faithful actions on $S^1$, so only faithful actions are interesting for our purposes.

\section{Thompson's groups}
Let $F$ and $T$ be the Thompson's groups acting on the closed interval and on the circle, respectively~\cite{CFP1996}. Recall that these are the groups of piecewise linear homeomorphisms of the interval and circle respectively, with dyadic breakpoints and all slopes given by powers of two. It is known that the standard action of $T$ is conjugate to a $C^\infty$ action~\cite{GS1987}.
The restriction of such a smooth action yields a smooth action of $F$ on a closed interval.
On the other hand, the groups $F*\Z$ and $T\ast\bZ$ do not admit any smooth actions on the circle, or indeed on any compact one--manifold (cf. Corollary~\ref{cor:Thompson}):

\begin{cor}
If $M$ is a compact one-manifold,
then $F\ast\bZ$ and $T\ast\bZ$ are not subgroups of $\Diff^{1+\mathrm{bv}}(M)$.
\end{cor}

\bp
Since $F\le T$, it suffices to prove the corollary for $F*\Z$ only. In order to apply Theorem~\ref{thm:main}, it suffices to show that $F\times\Z\le F$ and that $F$ is not virtually metabelian, whence the conclusion will be immediate. These claims follow immediately from the well--known facts that $F$ is not virtually solvable and that $F\times F\le F$, and we make these details explicit below for the reader's convenience.

Since $F$ is a group of homeomorphisms of the interval, it is immediate that it is torsion--free. Moreover, conjugating $F$ by the homeomorphism of $\R$ given by $x\mapsto x/2$ scales $F$ to be the group of piecewise linear homeomorphisms with dyadic breakpoints and slopes given by powers of two, only scaled to act on the interval $[0,1/2]$. It follows that we may realize $F\times F\le F$, since we can realize one copy of $F$ on $[0,1/2]$ and a second one on the interval $[1/2,1]$, with the points $\{0,1/2,1\}$ globally invariant. It follows that $\Z\times F\le F$, since $F$ is torsion--free.

To see that $F$ is not virtually metabelian, we use the standard fact that $[F,F]$ is an infinite simple group and that $Z(F)=\{1\}$ (cf.~\cite{CFP1996}). It follows that if $H\le F$ is a finite index subgroup then $[F,F]\le H$. Since $H$ contains an infinite simple group, it cannot be solvable, much less metabelian.
\ep

\section*{Acknowledgements}
The authors thank the anonymous referees for many careful comments which improved the exposition of the paper. The first author is supported by Samsung Science and Technology Foundation (SSTF-BA1301-06).
The second author is partially supported by Simons Foundation Collaboration Grant number 429836.

\def\cprime{$'$} \def\soft#1{\leavevmode\setbox0=\hbox{h}\dimen7=\ht0\advance
  \dimen7 by-1ex\relax\if t#1\relax\rlap{\raise.6\dimen7
  \hbox{\kern.3ex\char'47}}#1\relax\else\if T#1\relax
  \rlap{\raise.5\dimen7\hbox{\kern1.3ex\char'47}}#1\relax \else\if
  d#1\relax\rlap{\raise.5\dimen7\hbox{\kern.9ex \char'47}}#1\relax\else\if
  D#1\relax\rlap{\raise.5\dimen7 \hbox{\kern1.4ex\char'47}}#1\relax\else\if
  l#1\relax \rlap{\raise.5\dimen7\hbox{\kern.4ex\char'47}}#1\relax \else\if
  L#1\relax\rlap{\raise.5\dimen7\hbox{\kern.7ex
  \char'47}}#1\relax\else\message{accent \string\soft \space #1 not
  defined!}#1\relax\fi\fi\fi\fi\fi\fi}
\providecommand{\bysame}{\leavevmode\hbox to3em{\hrulefill}\thinspace}
\providecommand{\MR}{\relax\ifhmode\unskip\space\fi MR }
\providecommand{\MRhref}[2]{%
  \href{http://www.ams.org/mathscinet-getitem?mr=#1}{#2}
}
\providecommand{\href}[2]{#2}

\end{document}